\documentclass[reqno]{amsart}\usepackage{amsmath}

\usepackage{amsfonts}
\usepackage{amssymb}
\usepackage{hyperref}

\begin{document}
\title[ Positive solutions ]{Existence of positive solutions for a
three-point integral boundary-value problem}
\author[F. Haddouchi, S. Benaicha]{Faouzi Haddouchi, Slimane Benaicha}
\address{Faouzi Haddouchi\\
Department of Physics, University of Sciences and Technology of
Oran, El Mnaouar, BP 1505, 31000 Oran, Algeria}
\email{fhaddouchi@gmail.com}
\address{Slimane Benaicha \\
Department of Mathematics, University of Oran, Es-senia, 31000 Oran,
Algeria} \email{slimanebenaicha@yahoo.fr}

\begin{abstract}
In this paper, by using the Krasnoselskii's fixed-point theorem, we study the existence of at
least one or two positive solutions to the three-point integral boundary value problem
\begin{equation*} \label{eq-1}
\begin{gathered}
{u^{\prime \prime }}(t)+a(t)f(u(t))=0,\ 0<t<T, \\
u(0)={\beta}u(\eta),\ u(T)={\alpha}\int_{0}^{\eta}u(s)ds,
\end{gathered}
\end{equation*}
where $0<{\eta}<T$, $0<{\alpha}< \frac{2T}{{\eta}^{2}}$,
$0\leq{\beta}<\frac{2T-\alpha\eta^{2}}{\alpha\eta^{2}-2\eta+2T}$ are
given constants.
\end{abstract}

\maketitle \numberwithin{equation}{section}
\newtheorem{theorem}{Theorem}[section]
\newtheorem{lemma}[theorem]{Lemma} \newtheorem{proposition}[theorem]{%
Proposition} \newtheorem{corollary}[theorem]{Corollary} \newtheorem{remark}[%
theorem]{Remark}
\newtheorem{exmp}{Example}[section]

\section{Introduction}

We are interested in the existence of positive solutions of the following three-point integral
boundary value problem (BVP):

\begin{equation} \label{eq-2}
{u^{\prime \prime }}(t)+a(t)f(u(t))=0,\ t\in(0,T),
\end{equation}
\begin{equation} \label{eq-3}
u(0)={\beta}u(\eta),\ u(T)={\alpha}\int_{0}^{\eta}u(s)ds,
\end{equation}
where $0<{\eta}<T$ and $0<{\alpha}<
\frac{2T}{{\eta}^{2}}$,
$0\leq{\beta}<\frac{2T-\alpha\eta^{2}}{\alpha\eta^{2}-2\eta+2T}$,
and
\begin{itemize}
\item[(B1)] $f\in C([0,\infty),[0,\infty))$; \item[(B2)] $a\in
C([0,T],[0,\infty))$ and there exists $t_{0}\in[\eta,T]$ such that
$a(t_{0})>0$.
\end{itemize}

Set
\begin{equation} \label{eq-4}
f_{0}=\lim_{u\rightarrow0^{+}}\frac{f(u)}{u}, \
f_{\infty}=\lim_{u\rightarrow\infty}\frac{f(u)}{u}.
\end{equation}

The study of the existence of solutions of multi-point boundary value
problems for linear second-order ordinary differential equations was
initiated by II'in and Moiseev \cite{Ilin}.
Then Gupta \cite{Gupt} studied three-point boundary value problems for nonlinear
second-order ordinary differential equations. Since then, the existence of positive
solutions for nonlinear second order three-point boundary-value problems has been
studied by many authors by using the fixed point theorem, nonlinear alternative of the Leray-Schauder
approach, or coincidence degree theory. We refer the reader to \cite{Cheng,Guo,Han,Li,Pang,Sun1,Sun2,Webb1,Webb2,Webb3,Webb4,Webb5,Webb6,Webb7,Xu,Ander1,Ander2,He,Ma1,Ma2,Ma3,Ma4,Ma5,Feng,Maran,Luo, Liang1,Liu1,Liu2} and the references therein.

Tariboon and Sitthiwirattham \cite{Tarib} proved the existence of at least one positive solution on the condition that $f$ is either
superlinear or sublinear for the following BVP
\begin{equation} \label{eq-5}
{u^{\prime \prime }}(t)+a(t)f(u(t))=0,\ t\in(0,1),
\end{equation}
\begin{equation} \label{eq-6}
u(0)=0,\ u(1)={\alpha}\int_{0}^{\eta}u(s)ds,
\end{equation}
where $0<{\eta}<1$ and $0<{\alpha}<
\frac{2}{{\eta}^{2}}$, $f\in C([0,\infty),[0,\infty))$, $a\in C([0,1],[0,\infty))$ and there exists $t_{0}\in[\eta,1]$ such that
$a(t_{0})>0$. Very recently, Haddouchi and Benaicha \cite{Haddou}, investigated the following three-point BVP

\begin{equation} \label{eq-7}
{u^{\prime \prime }}(t)+a(t)f(u(t))=0,\ t\in(0,T),
\end{equation}
\begin{equation} \label{eq-8}
u(0)={\beta}u(\eta),\ u(T)={\alpha}\int_{0}^{\eta}u(s)ds,
\end{equation}
where $0<{\eta}<T$ and $0<{\alpha}<
\frac{2T}{{\eta}^{2}}$,
$0\leq{\beta}<\frac{2T-\alpha\eta^{2}}{\alpha\eta^{2}-2\eta+2T}$,  $f\in C([0,\infty),[0,\infty))$, $a\in C([0,T],[0,\infty))$ and there exists $t_{0}\in[\eta,T]$ such that
$a(t_{0})>0$, and improved the results in \cite{Tarib}.

In \cite{Haddou}, the authors used the Krasnoselskii's theorem to prove the following
result:
\begin{theorem}[See \cite{Haddou}]\label{theo 1.1}
Assume {\rm (B1)} and {\rm (B2)} hold, and $0<{\alpha}<
\frac{2T}{{\eta}^{2}}$,
$0\leq{\beta}<\frac{2T-\alpha\eta^{2}}{\alpha\eta^{2}-2\eta+2T}$.
If either
\begin{itemize}
\item[(D1)] $f_{0}=0$ and $f_{\infty}=\infty$ ($f$ is
superlinear), or \item[(D2)]$f_{0}=\infty$ and $f_{\infty}=0$ ($f$
is sublinear)
\end{itemize}
then problem \eqref{eq-7}, \eqref{eq-8} has at least one positive
solution.
\end{theorem}

Liu \cite{Liu3} used the fixed-point index theorem to prove the existence of at least one or two positive solutions to the three-point boundary value problem BVP
\begin{equation} \label{eq-9}
{u^{\prime \prime }}(t)+a(t)f(u(t))=0,\ t\in(0,1),
\end{equation}
\begin{equation} \label{eq-10}
u(0)=0,\ u(1)=\beta u(\eta),
\end{equation}
where $0<{\eta}<1$ and $0<\beta<\frac{1}{\eta}$.

Recently, Liang et al.\cite{Liang2}, investigated the following three-point BVP

\begin{equation} \label{eq-11}
{u^{\prime \prime }}(t)+a(t)f(u(t))=0,\ t\in(0,T),
\end{equation}
\begin{equation} \label{eq-12}
u(0)=\beta u(\eta),\ u(T)=\alpha u(\eta),
\end{equation}
where $0<{\eta}<T$, $0<\alpha<\frac{T}{\eta}$, $0<\beta<\frac{T-\alpha\eta}{T-\eta}$ are given constants, and obtained some simple criterions for the existence of at least one or two positive solutions by applying the Krasnoselskii's fixed point theorem under certain conditions on $f$.

Motivated by the results of \cite{Haddou,Liu3,Liang2} the aim of this paper is to establish some results for the existence of positive solutions of the BVP \eqref{eq-2}, \eqref{eq-3}, under $f_{0}=f_{\infty}=\infty$ or $f_{0}=f_{\infty}=0$. We also obtain some existence results for positive
solutions of the BVP \eqref{eq-2}, \eqref{eq-3} under $f_{0}, f_{\infty}\not\in \left\{0,\infty\right\}$. Finally, we give some examples to illustrate our results.

The key tool in our approach is the
following Krasnoselskii's fixed point theorem in a cone \cite{Krasn}.

\begin{theorem}[\cite{Krasn}]\label{theo 1.2}
Let $E$ be a Banach space, and let $K\subset E$ be a cone. Assume
$\Omega_{1}$, $\Omega _{2}$ are open bounded subsets of $E$ with $0\in \Omega _{1}$%
, $\overline{\Omega }_{1}\subset \Omega _{2}$, and let

\begin{equation*}\label{eq-13}
A: K\cap (\overline{\Omega }_{2}\backslash  \Omega
_{1})\longrightarrow K
\end{equation*}

be a completely continuous operator such that either

(i) $\ \left\Vert Au\right\Vert \leq \left\Vert u\right\Vert $, $\
u\in K\cap \partial \Omega _{1}$, \ and $\left\Vert Au\right\Vert
\geq \left\Vert u\right\Vert $, $\ u\in K\cap \partial \Omega _{2}$;
or

(ii) $\left\Vert Au\right\Vert \geq \left\Vert u\right\Vert $, $\
u\in K\cap
\partial \Omega _{1}$, \ and $\left\Vert Au\right\Vert \leq \left\Vert
u\right\Vert $, $\ u\in K\cap \partial \Omega _{2}$

hold. Then $A$ has a fixed point in $K\cap
(\overline{\Omega}_{2}\backslash $ $\Omega _{1})$.
\end{theorem}

\section{Preliminaries}
To prove the main existence results we will employ several
straightforward lemmas.

\begin{lemma}[See \cite{Haddou}]\label{lem 2.1}
Let $\beta \neq \frac{2T-\alpha \eta ^{2}}{\alpha \eta ^{2}-2\eta
+2T}$. Then for $y\in C([0,T],\mathbb{R})$, the problem
\begin{equation}\label{eq-14}
{u^{\prime \prime }}(t)+y(t)=0, \ t\in (0,T),
\end{equation}
\begin{equation}\label{eq-15}
u(0)=\beta u(\eta ), \ u(T)=\alpha \int_{0}^{\eta }u(s)ds
\end{equation}
has a unique solution
\begin{eqnarray*}
u(t)&=&\frac{\beta (2T-\alpha \eta ^{2})-2\beta (1-\alpha \eta
)t}{(\alpha \eta ^{2}-2T)-\beta (2\eta -\alpha \eta
^{2}-2T)}\int_{0}^{\eta }(\eta -s)y(s)ds \\
&&+\frac{\alpha \beta
\eta -\alpha (\beta -1)t}{(\alpha \eta ^{2}-2T)-\beta (2\eta -\alpha
\eta
^{2}-2T)}\int_{0}^{\eta }(\eta -s)^{2}y(s)ds \\
&&+\frac{2(\beta-1)t-2\beta \eta }{(\alpha \eta
^{2}-2T)-\beta (2\eta -\alpha \eta ^{2}-2T)}\int_{0}^{T}(T-s)y(s)ds-%
\int_{0}^{t}(t-s)y(s)ds.
\end{eqnarray*}
\end{lemma}
\begin{lemma}[See \cite{Haddou}]\label{lem 2.2}
Let $0<\alpha <\frac{2T}{\eta ^{2}}$, $\ 0\leq \beta <%
\frac{2T-\alpha \eta ^{2}}{\alpha \eta ^{2}-2\eta +2T}$. If $y\in C([0,T],[0,\infty
))$, then the unique solution $u$ of problem
\eqref{eq-14}, \eqref{eq-15} satisfies $\ u(t)\geq 0$ for $t\in [0,T]$.
\end{lemma}
\begin{remark} In view of Lemma 2.3 of  \cite{Haddou}, if $\ \alpha >\frac{2T}{\eta ^{2}}$, $\beta\geq0$ and $y\in C([0,T],[0,\infty
))$, then problem \eqref{eq-14}, \eqref{eq-15} has no positive solution. Hence, in this paper, we assume that $\ \alpha {\eta ^{2}}<2T$ and $\ 0\leq\beta <\frac{2T-\alpha \eta ^{2}}{\alpha \eta ^{2}-2\eta +2T}$ .
\end{remark}
\begin{lemma}[See \cite{Haddou}]\label{lem 2.4}
Let $0<\alpha <\frac{2T}{\eta ^{2}}$, $\ 0\leq\beta <%
\frac{2T-\alpha \eta ^{2}}{\alpha \eta ^{2}-2\eta +2T}$. If $y\in
C([0,T],[0,\infty ))$, then the unique solution $u$ of the problem
\eqref{eq-14}, \eqref{eq-15} satisfies
\begin{equation} \label{eq-16}
\min_{t\in [\eta,T]}u(t)\geq \gamma \|u\|,\ \|u\|=\max_{t\in
[0,T]}|u(t)|,
\end{equation}
where
\begin{equation} \label{eq-17}
\gamma:=\min\left\{\frac{\eta}{T},
\frac{\alpha(\beta+1)\eta^{2}}{2T},
\frac{\alpha(\beta+1)\eta(T-\eta)}{2T-\alpha(\beta+1)\eta^{2}}\right\}\in\left(0,1\right).
\end{equation}
\end{lemma}

In the rest of this article, we assume that $0<\alpha <\frac{2T}{\eta ^{2}}$, $\ 0\leq\beta <%
\frac{2T-\alpha \eta ^{2}}{\alpha \eta ^{2}-2\eta +2T}$. Let $E = C([0,T],\mathbb{R})$, and only the sup norm is used. It is easy to see that the BVP \eqref{eq-2}, \eqref{eq-3} has a solution  $u = u(t)$ if and only if  $u$ is a fixed point
of operator $A$, where $A$ is defined by
\begin{equation}\label{eq-18}
\begin{split}
Au(t)&=\frac{\beta (2T-\alpha \eta ^{2})-2\beta (1-\alpha \eta
)t}{(\alpha \eta ^{2}-2T)-\beta (2\eta -\alpha \eta
^{2}-2T)}\int_{0}^{\eta }(\eta -s)a(s)f(u(s))ds \\
&+\frac{\alpha \beta \eta -\alpha (\beta -1)t}{(\alpha \eta
^{2}-2T)-\beta (2\eta -\alpha \eta
^{2}-2T)}\int_{0}^{\eta }(\eta -s)^{2}a(s)f(u(s))ds \\
&+\frac{2(\beta-1)t-2\beta \eta }{(\alpha \eta ^{2}-2T)-\beta
(2\eta -\alpha \eta ^{2}-2T)}\int_{0}^{T}(T-s)a(s)f(u(s))ds \\
&-\int_{0}^{t}(t-s)a(s)f(u(s))ds.
\end{split}
\end{equation}

Denote
\begin{equation}\label{eq-19}
K=\left\{u\in E: u\geq0, \min_{t\in
[\eta,T]}u(t)\geq \gamma \|u\|\right\},
\end{equation}
where $\gamma$ is defined in \eqref{eq-17}.
It is obvious that $K$ is a cone in $E$. Moreover, by Lemma \ref{lem 2.2} and Lemma \ref{lem 2.4}, $AK\subset
K$. It is also easy to check that $A:K\rightarrow K$ is completely
continuous.\\

In what follows, for the sake of convenience, set
\begin{equation*}\label{eq-20}
\Lambda_{1}=\frac{(2T-\alpha\eta^{2})-\beta(\alpha\eta^{2}-2\eta+2T)}{\left[2(\beta+1)+T^{-1}\beta\eta(\alpha\eta+2)+\alpha\beta
T\right] \int_{0}^{T}T(T-s)a(s)ds},
\end{equation*}
\begin{equation*}\label{eq-21}
\Lambda_{2}=\frac{(2T-\alpha\eta^{2})-\beta(\alpha\eta^{2}-2\eta+2T)}{2\gamma\eta\int_{\eta}^{T}(T-s)a(s)ds}.
\end{equation*}
\section{The existence results of the BVP \eqref{eq-2}, \eqref{eq-3} for the case: $ f_{0}=f_{\infty}=\infty $ or $f_{0}=f_{\infty}=0 $}

Now we establish conditions for the existence of positive
solutions for the BVP \eqref{eq-2}, \eqref{eq-3} under
$f_{0}=f_{\infty}=\infty$ or  $f_{0}=f_{\infty}=0$.

\begin{theorem}\label{theo 3.1}
Assume that the following assumptions are satisfied.
\begin{itemize}
\item[(H1)]
$f_{0}=f_{\infty}=\infty$.
\item[(H2)] There exist constants $\rho_{1}> 0$ and $M_{1}\in(0,\Lambda_{1}]$ such that
$f(u)\leq M_{1}\rho_{1}$, for $u\in[0,\rho_{1}]$ .
\end{itemize}
Then, the problem \eqref{eq-2}, \eqref{eq-3} has at least two positive solutions $u_{1}$ and $u_{2}$
such that
\begin{equation*}\label{eq-22}
0<\|u_{1}\|<\rho_{1}<\|u_{2}\|.
\end{equation*}
\end{theorem}
\begin{proof}
Since, $f_{0}=\infty$, then for any
$M_{\star}\in\left[\Lambda_{2},\infty\right)$, there exists
$\rho_{\star}\in\left(0,\rho_{1}\right)$ such that $f(u)\geq
M_{\star}u$,  $0< u\leq \rho_{\star}$.

Set $\Omega_{\rho_{\star}}=\left\{u\in E: \|u\|<\rho_{\star}\right\}$. By \eqref{eq-18} and in view of the proof of Theorem 3.1 in \cite{Haddou}, for any $u\in K\cap \partial\Omega_{\rho_{\star}}$, we obtain
\begin{eqnarray*}
Au(\eta)
&=&\frac{2\eta}{(2T-\alpha\eta^{2})-\beta(\alpha\eta^{2}-2\eta+2T)}\int_{0}^{T}(T-s)a(s)f(u(s))ds \\
&&-\frac{\alpha \eta}{(2T-\alpha\eta^{2})-\beta(\alpha\eta^{2}-2\eta+2T)}\int_{0}^{\eta }(\eta^{2}-2\eta s+s^{2})a(s)f(u(s))ds \\
&&-\frac{2T-\alpha
\eta^{2}}{(2T-\alpha\eta^{2})-\beta(\alpha\eta^{2}-2\eta+2T)}\int_{0}^{\eta}(\eta
-s)a(s)f(u(s))ds
\end{eqnarray*}

\begin{eqnarray*}
&=&\frac{2\eta}{(2T-\alpha\eta^{2})-\beta(\alpha\eta^{2}-2\eta+2T)}\int_{\eta}^{T}(T-s)a(s)f(u(s))ds \\
&&+\frac{2(T-\eta)}{(2T-\alpha\eta^{2})-\beta(\alpha\eta^{2}-2\eta+2T)}\int_{0}^{\eta}sa(s)f(u(s))ds \\
&&+\frac{\alpha\eta}{(2T-\alpha\eta^{2})-\beta(\alpha\eta^{2}-2\eta+2T)}\int_{0}^{\eta}s(\eta-s)a(s)f(u(s))ds \\
&\geq&\frac{2\eta}{(2T-\alpha\eta^{2})-\beta(\alpha\eta^{2}-2\eta+2T)}\int_{\eta}^{T}(T-s)a(s)f(u(s))ds \\
&\geq&\rho_{\star}\gamma M_{\star}\frac{2\eta}{(2T-\alpha\eta^{2})-\beta(\alpha\eta^{2}-2\eta+2T)}\int_{\eta}^{T}(T-s)a(s)ds \\
&=&\rho_{\star}M_{\star}\Lambda_{2}^{-1}\\
&\geq&\rho_{\star}=\|u\|.
\end{eqnarray*}
Thus
\begin{equation}\label{eq-23}
\|Au\|\geq\|u\|, \ \ \text{for}\ u\in K\cap\partial\Omega_{\rho_{\star}}.
\end{equation}

Now, since $f_{\infty}=\infty$, then for any $M^{\star}\in\left[\Lambda_{2},\infty\right)$, there exists $\rho^{\star}>\rho_{1}$ such that $f(u)\geq M^{\star}u$, for $u\geq \gamma\rho^{\star}$.

Set $\Omega_{\rho^{\star}}=\left\{u\in E: \|u\|<\rho^{\star}\right\}$.
Then, for any $u\in K\cap \partial\Omega_{\rho^{\star}}$, we have
\begin{eqnarray*}
Au(\eta)&\geq&\frac{2\eta}{(2T-\alpha\eta^{2})-\beta(\alpha\eta^{2}-2\eta+2T)}\int_{\eta}^{T}(T-s)a(s)f(u(s))ds \\
&\geq&\rho^{\star}\gamma M^{\star}\frac{2\eta}{(2T-\alpha\eta^{2})-\beta(\alpha\eta^{2}-2\eta+2T)}\int_{\eta}^{T}(T-s)a(s)ds \\
&=&\rho^{\star}M^{\star}\Lambda_{2}^{-1}\\
&\geq&\rho^{\star}=\|u\|.
\end{eqnarray*}
Which implies
\begin{equation}\label{eq-24}
 \|Au\|\geq\|u\|,\ \ \text{for}\  u\in K\cap\partial\Omega_{\rho^{\star}}.
\end{equation}

Finally, set $\Omega_{\rho_{1}}=\left\{u\in E: \|u\|<\rho_{1}\right\}$. From {\rm (H2)}, \eqref{eq-18} and the proof of Theorem 3.1 in \cite{Haddou}, for any $u\in K\cap\partial\Omega_{\rho_{1}}$, we have
\begin{eqnarray*}
Au(t)
&\leq&\frac{2\beta T+\alpha\beta \eta ^{2}}{(2T-\alpha\eta^{2})-\beta(\alpha\eta^{2}-2\eta+2T)}\int_{0}^{\eta }(\eta -s)a(s)f(u(s))ds \\
&&+\frac{\alpha\beta T}{(2T-\alpha\eta^{2})-\beta(\alpha\eta^{2}-2\eta+2T)}\int_{0}^{\eta }(\eta -s)^{2}a(s)f(u(s))ds \\
&&+\frac{2\beta \eta+2T}{(2T-\alpha\eta^{2})-\beta(\alpha\eta^{2}-2\eta+2T)}\int_{0}^{T}(T-s)a(s)f(u(s))ds
\end{eqnarray*}

\begin{eqnarray*}
&\leq&\frac{2T(\beta+1)+\beta\eta(\alpha\eta+2)}{(2T-\alpha\eta^{2})-\beta(\alpha\eta^{2}-2\eta+2T)}\int_{0}^{T}(T-s)a(s)f(u(s))ds \\
&&+\frac{\alpha\beta T}{(2T-\alpha\eta^{2})-\beta(\alpha\eta^{2}-2\eta+2T)}\int_{0}^{T}T(T-s)a(s)f(u(s))ds \\
&=&\frac{2(\beta+1)+T^{-1}\beta\eta(\alpha\eta+2)+\alpha\beta
T}{(2T-\alpha\eta^{2})-\beta(\alpha\eta^{2}-2\eta+2T)}\int_{0}^{T}T(T-s)a(s)f(u(s))ds\\
&\leq& M_{1}\rho_{1}\frac{2(\beta+1)+T^{-1}\beta\eta(\alpha\eta+2)+\alpha\beta
T}{(2T-\alpha\eta^{2})-\beta(\alpha\eta^{2}-2\eta+2T)}\int_{0}^{T}T(T-s)a(s)ds\\
&=&\rho_{1}M_{1}\Lambda_{1}^{-1}\leq\rho_{1}=\|u\|.
\end{eqnarray*}
Which yields
\begin{equation}\label{eq-25}
\|Au\|\leq\|u\|,\ \ \text{for}\ u\in K\cap\partial\Omega_{\rho_{1}}.
\end{equation}

Hence, since $\rho_{\star}<\rho_{1}<\rho^{\star}$ and from \eqref{eq-23}, \eqref{eq-24}, \eqref{eq-25}, it follows from Theorem \ref{theo 1.2} that $A$ has a fixed point $u_{1}$ in
$K\cap (\overline{\Omega}_{\rho_{1}}\backslash \Omega _{\rho_{\star}})$ and a fixed point $u_{2}$ in $K\cap (\overline{\Omega}_{\rho^{\star}}\backslash \Omega _{\rho_{1}})$. Both are positive solutions of the BVP \eqref{eq-2}, \eqref{eq-3} and $0<\|u_{1}\|<\rho_{1}<\|u_{2}\|$.
The proof is therefore complete.
\end{proof}
\begin{theorem}\label{theo 3.2}
Assume that the following assumptions are satisfied.
\begin{itemize}
\item[(H3)]
 $f_{0}=f_{\infty}=0$.
\item[(H4)] There exist constants $\rho_{2}> 0$ and $M_{2}\in\left[\Lambda_{2},\infty\right)$  such that
$f(u)\geq M_{2}\rho_{2}$, for $u\in[\gamma\rho_{2} ,\rho_{2}]$ .
\end{itemize}
Then, the problem \eqref{eq-2}, \eqref{eq-3} has at least two positive solutions $u_{1}$ and $u_{2}$
such that
\begin{equation*}\label{eq-26}
0<\|u_{1}\|<\rho_{2}<\|u_{2}\|.
\end{equation*}
\end{theorem}

\begin{proof}
Firstly, since $f_{0}=0$, for any
$\epsilon\in\left(0,\Lambda_{1}\right]$, there exists
$\rho_{\star}\in\left(0,\rho_{2}\right)$ such that
$f(u)\leq\epsilon u $, for  $u\in\left(0,\rho_{\star}\right]$. Let
$\Omega_{\rho_{\star}}=\left\{u\in E: \|u\|<\rho_{\star}\right\}$,
then, for any $u\in K\cap \partial\Omega_{\rho_{\star}}$, we
obtain
\begin{eqnarray*}
Au(t)&\leq&\frac{2(\beta+1)+T^{-1}\beta\eta(\alpha\eta+2)+\alpha\beta
T}{(2T-\alpha\eta^{2})-\beta(\alpha\eta^{2}-2\eta+2T)}\int_{0}^{T}T(T-s)a(s)f(u(s))ds\\
&\leq& \rho_{\star} \epsilon \frac{2(\beta+1)+T^{-1}\beta\eta(\alpha\eta+2)+\alpha\beta
T}{(2T-\alpha\eta^{2})-\beta(\alpha\eta^{2}-2\eta+2T)}\int_{0}^{T}T(T-s)a(s)ds\\
&=&\rho_{\star} \epsilon \Lambda_{1}^{-1}\leq\rho_{\star}=\|u\|,
\end{eqnarray*}
which implies
\begin{equation}\label{eq-27}
\|Au\|\leq\|u\|,\ \ \text{for}\ u\in K\cap\partial\Omega_{\rho_{\star}}.
\end{equation}
Secondly, in view of $f_{\infty}=0$, for any $\epsilon_{1}\in\left(0,\Lambda_{1}\right]$, there exists $\rho_{0}>\rho_{2}$ such that
\begin{equation}\label{eq-28}
f(u)\leq\epsilon_{1} u , \ \ \text{for}\  u\in\left[\rho_{0},\infty\right).
\end{equation}

We consider two cases:

Case (i). Suppose that $f(u)$ is unbounded. Then from $f\in C([0,\infty),[0,\infty))$, we know that there
is $\rho^{\star}>\rho_{0}$ such that
\begin{equation}\label{eq-29}
f(u)\leq f(\rho^{\star}),   \ \ \text{for}\  u\in\left[0,\rho^{\star}\right].
\end{equation}
Since $\rho^{\star}>\rho_{0}$, then from \eqref{eq-28}, \eqref{eq-29}, one has
\begin{equation}\label{eq-30}
f(u)\leq f(\rho^{\star})\leq\epsilon_{1}\rho^{\star},   \ \ \text{for} \  u\in\left[0,\rho^{\star}\right].
\end{equation}
For $u\in K$ and $ \|u\|=\rho^{\star}$ , from \eqref{eq-30}, we obtain
\begin{eqnarray*}
Au(t)&\leq&\frac{2(\beta+1)+T^{-1}\beta\eta(\alpha\eta+2)+\alpha\beta
T}{(2T-\alpha\eta^{2})-\beta(\alpha\eta^{2}-2\eta+2T)}\int_{0}^{T}T(T-s)a(s)f(u(s))ds\\
&\leq& \rho^{\star} \epsilon_{1}\frac{2(\beta+1)+T^{-1}\beta\eta(\alpha\eta+2)+\alpha\beta
T}{(2T-\alpha\eta^{2})-\beta(\alpha\eta^{2}-2\eta+2T)}\int_{0}^{T}T(T-s)a(s)ds\\
&=&\rho^{\star} \epsilon_{1}\Lambda_{1}^{-1}\leq\rho^{\star}=\|u\|,
\end{eqnarray*}

Case (ii). Suppose that $f(u)$ is bounded, say $f(u)\leq L$ for all $u\in \left[0,\infty\right)$.
Taking $\rho^{\star}\geq \max\left\{\frac{L}{\epsilon_{1}}, \rho_{0}\right\}$.
For $u\in K$ with $ \|u\|=\rho^{\star}$, we have
\begin{eqnarray*}
Au(t)&\leq&\frac{2(\beta+1)+T^{-1}\beta\eta(\alpha\eta+2)+\alpha\beta
T}{(2T-\alpha\eta^{2})-\beta(\alpha\eta^{2}-2\eta+2T)}\int_{0}^{T}T(T-s)a(s)f(u(s))ds\\
&\leq& L \frac{2(\beta+1)+T^{-1}\beta\eta(\alpha\eta+2)+\alpha\beta
T}{(2T-\alpha\eta^{2})-\beta(\alpha\eta^{2}-2\eta+2T)}\int_{0}^{T}T(T-s)a(s)ds\\
&\leq& \rho^{\star}
\epsilon_{1}\frac{2(\beta+1)+T^{-1}\beta\eta(\alpha\eta+2)+\alpha\beta
T}{(2T-\alpha\eta^{2})-\beta(\alpha\eta^{2}-2\eta+2T)}\int_{0}^{T}T(T-s)a(s)ds\\
&=&\rho^{\star} \epsilon_{1}\Lambda_{1}^{-1}\leq\rho^{\star}=\|u\|.
\end{eqnarray*}
Hence, in either case, we always may set $\Omega_{\rho^{\star}}=\left\{u\in E: \|u\|<\rho^{\star}\right\}$ such that
\begin{equation}\label{eq-31}
\|Au\|\leq\|u\|,\ \ \text{for}\ u\in K\cap\partial\Omega_{\rho^{\star}}.
\end{equation}
Finally, set $\Omega_{\rho_{2}}=\left\{u\in E: \|u\|<\rho_{2}\right\}$. By {\rm (H4)}, for any $u\in K\cap\partial\Omega_{\rho_{2}}$, we can get
\begin{eqnarray*}
Au(\eta)&\geq&\frac{2\eta}{(2T-\alpha\eta^{2})-\beta(\alpha\eta^{2}-2\eta+2T)}\int_{\eta}^{T}(T-s)a(s)f(u(s))ds \\
&\geq&\frac{2\eta}{(2T-\alpha\eta^{2})-\beta(\alpha\eta^{2}-2\eta+2T)}\int_{\eta}^{T}(T-s)a(s)M_{2}\rho_{2}ds \\
&\geq&\rho_{2}\frac{2\eta M_{2}\gamma}{(2T-\alpha\eta^{2})-\beta(\alpha\eta^{2}-2\eta+2T)}\int_{\eta}^{T}(T-s)a(s)ds \\
&=&\rho_{2}M_{2}\Lambda_{2}^{-1}\\
&\geq&\rho_{2}=\|u\|,
\end{eqnarray*}
which implies
\begin{equation}\label{eq-32}
\|Au\|\geq\|u\|,\ \ \text{for}\ u\in K\cap\partial\Omega_{\rho_{2}}.
\end{equation}

Hence, since $\rho_{\star}<\rho_{2}<\rho^{\star}$ and from \eqref{eq-27}, \eqref{eq-31} and \eqref{eq-32}, it follows from Theorem \ref{theo 1.2} that $A$ has a fixed point $u_{1}$ in $K\cap (\overline{\Omega}_{\rho_{2}}\backslash \Omega _{\rho_{\star}})$ and a fixed point $u_{2}$ in $K\cap (\overline{\Omega}_{\rho^{\star}}\backslash \Omega _{\rho_{2}})$. Both are positive solutions of the BVP \eqref{eq-2}, \eqref{eq-3} and $0<\|u_{1}\|<\rho_{2}<\|u_{2}\|$.
The proof is therefore complete.
\end{proof}

\section{The existence results of the BVP \eqref{eq-2}, \eqref{eq-3} for the case: $f_{0}, f_{\infty}\not\in \left\{0,\infty\right\}$}

In this section, we discuss the existence for the positive solution of the BVP \eqref{eq-2}, \eqref{eq-3} assuming $f_{0}, f_{\infty}\not\in \left\{0,\infty\right\}$.

Now, we shall state and prove the following main result.
\begin{theorem}\label{theo 4.1}
Suppose {\rm (H2)} and {\rm (H4)} hold and that $\rho_{1}\neq \rho_{2}$. Then, the BVP \eqref{eq-2}, \eqref{eq-3} has
at least one positive solution $u$ satisfying $\rho_{1}<\|u\|<\rho_{2}$ or $\rho_{2}<\|u\|<\rho_{1}$.
\end{theorem}
\begin{proof}
Without loss of generality, we may assume that $\rho_{1}<\rho_{2}$.

Let $\Omega_{\rho_{1}}=\left\{u\in E: \|u\|<\rho_{1}\right\}$. By {\rm (H2)}, for any $u\in
K\cap\partial\Omega_{\rho_{1}}$, we obtain
\begin{eqnarray*}
Au(t)&\leq& \frac{2(\beta+1)+T^{-1}\beta\eta(\alpha\eta+2)+\alpha\beta
T}{(2T-\alpha\eta^{2})-\beta(\alpha\eta^{2}-2\eta+2T)}\int_{0}^{T}T(T-s)a(s)f(u(s))ds\\
&\leq& M_{1}\rho_{1} \frac{2(\beta+1)+T^{-1}\beta\eta(\alpha\eta+2)+\alpha\beta
T}{(2T-\alpha\eta^{2})-\beta(\alpha\eta^{2}-2\eta+2T)}\int_{0}^{T}T(T-s)a(s)ds\\
&=&\rho_{1}M_{1}\Lambda_{1}^{-1}\leq\rho_{1}=\|u\|,
\end{eqnarray*}
which yields
\begin{equation}\label{eq-33}
\|Au\|\leq\|u\| , \ u\in K\cap\partial\Omega_{\rho_{1}}.
\end{equation}

Now, set $\Omega_{\rho_{2}}=\left\{u\in E: \|u\|<\rho_{2}\right\}$. By {\rm (H4)}, for any $u\in K\cap\partial\Omega_{\rho_{2}}$, we can get
\begin{eqnarray*}
Au(\eta)&\geq&\frac{2\eta}{(2T-\alpha\eta^{2})-\beta(\alpha\eta^{2}-2\eta+2T)}\int_{\eta}^{T}(T-s)a(s)f(u(s))ds \\
&\geq&\frac{2\eta}{(2T-\alpha\eta^{2})-\beta(\alpha\eta^{2}-2\eta+2T)}\int_{\eta}^{T}(T-s)a(s)M_{2}\rho_{2}ds \\
&\geq&\rho_{2}\frac{2\eta M_{2}\gamma}{(2T-\alpha\eta^{2})-\beta(\alpha\eta^{2}-2\eta+2T)}\int_{\eta}^{T}(T-s)a(s)ds \\
&=&\rho_{2}M_{2}\Lambda_{2}^{-1}\\
&\geq&\rho_{2}=\|u\|,
\end{eqnarray*}
which implies
\begin{equation}\label{eq-34}
\|Au\|\geq\|u\|,\  \text{for} \ u\in K\cap\partial\Omega_{\rho_{2}}.
\end{equation}
Hence, since $\rho_{1}<\rho_{2}$ and from \eqref{eq-33} and \eqref{eq-34}, it follows from Theorem \ref{theo 1.2} that
$A$ has a fixed point $u$ in $K\cap (\overline{\Omega}_{\rho_{2}}\backslash \Omega _{\rho_{1}})$. Moreover, it is a positive solution
of the BVP \eqref{eq-2}, \eqref{eq-3} and
\begin{equation*}\label{eq-35}
\rho_{1}<\|u\|<\rho_{2}.
\end{equation*}
The proof is therefore complete.
\end{proof}
\begin{corollary}\label{cor 4.2}
Assume that the following assumptions hold.
\begin{itemize}
\item[(H5)] $f_{0}=\alpha_{1}\in
\left[0,\theta_{1}\Lambda_{1}\right)$, where
$\theta_{1}\in\left(0,1\right]$. \item[(H6)]
$f_{\infty}=\beta_{1}\in
\left(\frac{\theta_{2}}{\gamma}\Lambda_{2},\infty\right)$, where
$\theta_{2}\geq1$.
\end{itemize}
Then, the BVP \eqref{eq-2}, \eqref{eq-3} has at least one positive solution.
\end{corollary}
\begin{proof}
In view of $f_{0}=\alpha_{1}\in \left[0,\theta_{1}\Lambda_{1}\right)$, for $\epsilon=\theta_{1}\Lambda_{1}-\alpha_{1}>0$, there exists
a sufficiently large $\rho_{1}>0$ such that
\begin{equation*}\label{eq-36}
f(u)\leq(\alpha_{1}+\epsilon)u=\theta_{1}\Lambda_{1} u\leq\theta_{1}\Lambda_{1}\rho_{1},\ \text{for}\ u\in\left(0,\rho_{1}\right].
\end{equation*}

Since $\theta_{1}\in\left(0,1\right]$, then
$\theta_{1}\Lambda_{1}\in\left(0,\Lambda_{1}\right]$. By the
inequality above, {\rm (H2)} is satisfied. Since
$f_{\infty}=\beta_{1}\in
\left(\frac{\theta_{2}}{\gamma}\Lambda_{2},\infty\right)$, for
$\epsilon=\beta_{1}-\frac{\theta_{2}}{\gamma}\Lambda_{2}>0$, there
exists a sufficiently large $\rho_{2}(>\rho_{1})$ such that
\begin{equation*}\label{eq-37}
\frac{f(u)}{u}\geq\beta_{1}-\epsilon=\frac{\theta_{2}}{\gamma}\Lambda_{2},\ \text{for}\ u\in\left[\gamma\rho_{2},\infty\right),
\end{equation*}
thus, when $u\in\left[\gamma\rho_{2},\rho_{2}\right]$, one has
\begin{equation*}\label{eq-38}
f(u)\geq\frac{\theta_{2}}{\gamma}\Lambda_{2}u\geq\theta_{2}\Lambda_{2}\rho_{2}.
\end{equation*}
Since $\theta_{2}\geq1$, $\theta_{2}\Lambda_{2}\in\left[\Lambda_{2},\infty\right)$, then from the above inequality, condition {\rm
(H4)} is satisfied.
Hence, from Theorem \ref{theo 4.1} , the desired result holds.
\end{proof}
\begin{corollary}\label{cor 4.3}
Assume that the following assumptions hold.
\begin{itemize}
\item[(H7)] $f_{0}=\alpha_{2}\in \left(\frac{\theta_{2}}{\gamma}\Lambda_{2},\infty\right)$, where $\theta_{2}\geq1$.
\item[(H8)] $f_{\infty}=\beta_{2}\in \left[0,\theta_{1}\Lambda_{1}\right)$, where $\theta_{1}\in\left(0,1\right]$.
\end{itemize}
Then, the BVP \eqref{eq-2}, \eqref{eq-3} has at least one positive solution.
\end{corollary}
\begin{proof}
Since $f_{0}=\alpha_{2}\in \left(\frac{\theta_{2}}{\gamma}\Lambda_{2},\infty\right)$, for
$\epsilon=\alpha_{2}-\frac{\theta_{2}}{\gamma}\Lambda_{2}>0$, there exists a sufficiently
small $\rho_{2}>0$ such that
\begin{equation*}\label{eq-39}
\frac{f(u)}{u}\geq\alpha_{2}-\epsilon=\frac{\theta_{2}}{\gamma}\Lambda_{2},\ \text{for}\ u\in\left(0,\rho_{2}\right].
\end{equation*}
Thus, when $u\in\left[\gamma\rho_{2},\rho_{2}\right]$, one has
\begin{equation*}\label{eq-40}
f(u)\geq\frac{\theta_{2}}{\gamma}\Lambda_{2}u\geq\theta_{2}\Lambda_{2}\rho_{2}.
\end{equation*}
which yields the condition {\rm (H4)} of Theorem \ref{theo 3.2}.

In view of $f_{\infty}=\beta_{2}\in \left[0,\theta_{1}\Lambda_{1}\right)$, for $\epsilon=\theta_{1}\Lambda_{1}-\beta_{2}>0$, there
exists a sufficiently large $\rho_{0}(>\rho_{2})$ such that
\begin{equation}\label{eq-41}
\frac{f(u)}{u}\leq\beta_{2}+\epsilon=\theta_{1}\Lambda_{1},\ \text{for}\ u\in\left[\rho_{0},\infty\right).
\end{equation}
We consider the following two cases:

Case (i). Suppose that $f(u)$ is unbounded. Then from $f\in C([0,\infty),[0,\infty))$, we know that there
is $\rho_{1}>\rho_{0}$ such that
\begin{equation}\label{eq-42}
f(u)\leq f(\rho_{1}),   \ \ \text{for}\  u\in\left[0,\rho_{1}\right].
\end{equation}
Since $\rho_{1}>\rho_{0}$, then from \eqref{eq-41}, \eqref{eq-42}, one has
\begin{equation*}\label{eq-43}
f(u)\leq f(\rho_{1})\leq \theta_{1}\Lambda_{1}\rho_{1},   \ \ \text{for} \  u\in\left[0,\rho_{1}\right].
\end{equation*}
Since $\theta_{1}\in\left(0,1\right]$, then $\theta_{1}\Lambda_{1}\in\left(0,\Lambda_{1}\right]$. By the inequality above, {\rm (H2)}
is satisfied.

Case (ii). Suppose that $f(u)$ is bounded, say
\begin{equation}\label{eq-44}
f(u)\leq L, \ \  \text{for all} \ \ u\in \left[0,\infty\right)
\end{equation}

In this case, taking sufficiently large $\rho_{1}>\frac{L}{\theta_{1}\Lambda_{1}}$, then from \eqref{eq-44}, we know
\begin{equation*}\label{eq-45}
f(u)\leq L\leq\theta_{1}\Lambda_{1}\rho_{1}, \ \ \text{for} \  u\in\left[0,\rho_{1}\right].
\end{equation*}
Since $\theta_{1}\in\left(0,1\right]$, then $\theta_{1}\Lambda_{1}\in\left(0,\Lambda_{1}\right]$. By the inequality above, {\rm (H2)}
is satisfied.
Hence, from Theorem \ref{theo 4.1}, we get the conclusion of Corollary \ref{cor 4.3}.
\end{proof}

\begin{corollary}\label{cor 4.4}
Assume that the previous hypotheses {\rm (H2)}, {\rm (H6)} and {\rm (H7)} hold. Then, the BVP \eqref{eq-2}, \eqref{eq-3} has at least
two positive solutions $u_{1}$ and $u_{2}$ such that
\begin{equation*}\label{eq-46}
0<\|u_{1}\|<\rho_{1}<\|u_{2}\|.
\end{equation*}
\end{corollary}
\begin{proof}
From {\rm (H6)} and the proof of Corollary \ref{cor 4.2}, we know that there exists a sufficiently large $\rho_{2}>\rho_{1}$,  such
that
\begin{equation*}\label{eq-47}
f(u)\geq\theta_{2}\Lambda_{2}\rho_{2}=M_{2}\rho_{2},   \ \ \text{for}\  u\in\left[\gamma\rho_{2},\rho_{2}\right],
\end{equation*}
where $M_{2}=\theta_{2}\Lambda_{2}\in \left[\Lambda_{2},\infty\right)$.

In view of {\rm (H7)} and the proof of Corollary \ref{cor 4.3}, we see that there exists a sufficiently small
$\rho_{2}^{\star}\in\left(0,\rho_{1}\right)$ such that
\begin{equation*}\label{eq-48}
f(u)\geq\theta_{2}\Lambda_{2}\rho_{2}^{\star}=M_{2}\rho_{2}^{\star},   \ \ \text{for}\
u\in\left[\gamma\rho_{2}^{\star},\rho_{2}^{\star}\right],
\end{equation*}
where $M_{2}=\theta_{2}\Lambda_{2}\in \left[\Lambda_{2},\infty\right)$.

Using this and {\rm (H2)}, we know by Theorem \ref{theo 4.1} that the BVP \eqref{eq-2}, \eqref{eq-3} has two positive
solutions $u_{1}$ and $u_{2}$ such that
\begin{equation*}\label{eq-49}
\rho_{2}^{\star}<\|u_{1}\|<\rho_{1}<\|u_{2}\|<\rho_{2}.
\end{equation*}
Thus, the proof is complete.
\end{proof}
\begin{corollary}\label{cor 4.5}
Assume that the previous hypotheses {\rm (H4)}, {\rm (H5)} and {\rm (H8)} hold. Then, the BVP \eqref{eq-2}, \eqref{eq-3} has at least
two positive solutions $u_{1}$ and $u_{2}$ such that
\begin{equation*}\label{eq-50}
0<\|u_{1}\|<\rho_{2}<\|u_{2}\|.
\end{equation*}
\end{corollary}
\begin{proof}
By {\rm (H5)} and the proof of Corollary \ref{cor 4.2}, we obtain that there exists sufficiently small
$\rho_{1}\in\left(0,\rho_{2}\right)$ such that
\begin{equation*}\label{eq-51}
f(u)\leq\theta_{1}\Lambda_{1}\rho_{1}=M_{1}\rho_{1}, \ \ \text{for}\  u\in\left(0,\rho_{1}\right],
\end{equation*}
where $M_{1}=\theta_{1}\Lambda_{1}\in\left(0,\Lambda_{1}\right]$.

In view of {\rm (H8)} and the proof of Corollary \ref{cor 4.3}, there exists a sufficiently large $\rho^{\star}_{1}>\rho_{2}$ such
that
\begin{equation*}\label{eq-52}
f(u)\leq\theta_{1}\Lambda_{1}\rho^{\star}_{1}=M_{1}\rho^{\star}_{1}, \ \ \text{for}\  u\in\left[0,\rho^{\star}_{1}\right],
\end{equation*}
where $M_{1}=\theta_{1}\Lambda_{1}\in\left(0,\Lambda_{1}\right]$.

Using this and {\rm (H4)}, we see by Theorem \ref{theo 4.1} that the BVP \eqref{eq-2}, \eqref{eq-3} has two positive solutions
$u_{1}$ and $u_{2}$ such that
\begin{equation*}\label{eq-53}
\rho_{1}<\|u_{1}\|<\rho_{2}<\|u_{2}\|<\rho^{\star}_{1}.
\end{equation*}
This completes the proof.
\end{proof}

\bigskip
\section{Illustration}
In this section, we give some examples about the theoretical results.

\begin{exmp}
Consider the boundary value problem

\begin{equation}\label{eq-54}
{u^{\prime \prime }}(t)+\frac{5}{32}(2-t)^{3}(\frac{u^{\frac{1}{2}}}{2}+\frac{u^{2}}{32})=0, \  \ 0<t<2,
\end{equation}

\begin{equation}\label{eq-55}
u(0)=\frac{1}{30}u(1), \  \ u(2)=2 \int_{0}^{1}u(s)ds.
\end{equation}
Set $\beta=1/30$, $\alpha=2$, $\eta=1$, $T=2$, $a(t)=\frac{5}{32}(2-t)^{3}$, $f(u)=\frac{u^{\frac{1}{2}}}{2}+\frac{u^{2}}{32}$. We can show that $0<\alpha=2<4=2T/{\eta^{2}}$, $0<\beta=1/30<1/2=(2T-\alpha \eta ^{2})/(\alpha \eta ^{2}-2\eta
+2T)$. Since $f_{0}=f_{\infty}=\infty$, then {\rm (H1)} holds. Again $\Lambda_{1}=((2T-\alpha\eta^{2})-\beta(\alpha\eta^{2}-2\eta+2T))/((2(\beta+1)+T^{-1}\beta\eta(\alpha\eta+2)+\alpha\beta
T)\int_{0}^{T}T(T-s)a(s)ds)=7/17$, because $f(u)$ is monotone increasing function for $u\geq0$, taking $\rho_{1}=4$, $M_{1}=3/8\in(0,\Lambda_{1}]$, then when $u\in[0,\rho_{1}]$, we get

\begin{equation*}\label{eq-56}
f(u)\leq f(4)=3/2=M_{1}\rho_{1},
\end{equation*}
which implies {\rm (H2)} holds. Hence, by Theorem \ref{theo 3.1}, the BVP \eqref{eq-54}, \eqref{eq-55} has at least two positive solutions $u_{1}$ and $u_{2}$ such that
\begin{equation*}\label{eq-57}
0<\|u_{1}\|<4<\|u_{2}\|.
\end{equation*}
\end{exmp}

\begin{exmp}
Consider the boundary value problem
\begin{equation}\label{eq-58}
{u^{\prime \prime }}(t)+8e^{6}u^{2}e^{-u}=0, \  \ 0<t<\frac{3}{4},
\end{equation}

\begin{equation}\label{eq-59}
u(0)=\frac{1}{10}u(\frac{1}{4}), \  \ u(\frac{3}{4})=20 \int_{0}^{\frac{1}{4}}u(s)ds.
\end{equation}
Set $\beta=1/10$, $\alpha=20$, $\eta=1/4$, $T=3/4$, $a(t)\equiv8$, $f(u)=e^{6}u^{2}e^{-u}$.
We can show that $0<\alpha=20<24=2T/{\eta^{2}}$, $0<\beta=1/10<1/9=(2T-\alpha \eta ^{2})/(\alpha \eta^{2}-2\eta+2T)$, $\gamma=\min\{\eta/T, (\alpha(\beta+1)\eta^{2})/2T, (\alpha(\beta+1)\eta(T-\eta))/(2T-\alpha(\beta+1)\eta^{2})\}=\min\{1/3, 11/12, 22\}=1/3$. Since $f_{0}=f_{\infty}=0$, then {\rm (H3)} holds. Again
$\Lambda_{2}=((2T-\alpha\eta^{2})-\beta(\alpha\eta^{2}-2\eta+2T))/(2\gamma\eta\int_{\eta}^{T}(T-s)a(s)ds)=3/20$,
because $f(u)$ is monotone decreasing function for $u\geq2$, taking $\rho_{2}=6$, $M_{2}=6\in[\Lambda_{2},\infty)$, then when $u\in[\gamma\rho_{2},\rho_{2} ]$, we obtain

\begin{equation*}\label{eq-60}
f(u)\geq f(6)=36=M_{2}\rho_{2},
\end{equation*}

which implies {\rm (H4)} holds. Hence, by Theorem \ref{theo 3.2}, the BVP \eqref{eq-58}, \eqref{eq-59} has at least two positive solutions $u_{1}$ and $u_{2}$ such that
\begin{equation*}\label{eq-61}
0<\|u_{1}\|<6<\|u_{2}\|.
\end{equation*}
\end{exmp}

\begin{exmp}
Consider the boundary value problem
\begin{equation}\label{eq-62}
{u^{\prime \prime }}(t)+\frac{aue^{2u}}{b+e^{u}+e^{2u}}=0, \  \ 0<t<1,
\end{equation}

\begin{equation}\label{eq-63}
u(0)=\frac{1}{2}u(\frac{1}{3}), \  \ u(1)=3 \int_{0}^{\frac{1}{3}}u(s)ds,
\end{equation}
where $a=183$, $b=637$. Set $\beta=1/2$, $\alpha=3$, $\eta=1/3$, $T=1$, $a(t)\equiv1$, $f(u)=(aue^{2u})/(b+e^{u}+e^{2u})$. We can show that $0<\alpha=3<18=2T/{\eta^{2}}$, $0<\beta=1/2<1=(2T-\alpha \eta ^{2})/(\alpha \eta^{2}-2\eta+2T)$. Since $\gamma=\min\{\eta/T, (\alpha(\beta+1)\eta^{2})/2T, (\alpha(\beta+1)\eta(T-\eta))/(2T-\alpha(\beta+1)\eta^{2})\}=\min\{1/3, 1/4, 2/3\}=1/4$, $\Lambda_{1}=((2T-\alpha\eta^{2})-\beta(\alpha\eta^{2}-2\eta+2T))/((2(\beta+1)+T^{-1}\beta\eta(\alpha\eta+2)+\alpha\beta
T)\int_{0}^{T}T(T-s)a(s)ds)=1/3$, $\Lambda_{2}=((2T-\alpha\eta^{2})-\beta(\alpha\eta^{2}-2\eta+2T))/(2\gamma\eta\int_{\eta}^{T}(T-s)a(s)ds)=45/2$, and $f_{0}=a/(b+2)=61/213$, $f_{\infty}=a=183$. Taking $\theta_{1}\in(61/71, 1]$, $\theta_{2}\in[1, 2]$, thus $f_{0}\in(0, \theta_{1}\Lambda_{1})$, $f_{\infty}\in((\theta_{2}/\gamma)\Lambda_{2}, \infty)$, which imply {\rm (H5)} and {\rm (H6)} hold. Therefore, by Corollary \ref{cor 4.2}, the BVP \eqref{eq-62}, \eqref{eq-63} has at least one positive solution.
\end{exmp}

\begin{exmp}
Consider the boundary value problem
\begin{equation}\label{eq-64}
{u^{\prime \prime }}(t)+\frac{6}{25}tu(1+\frac{\lambda}{1+u^{2}})=0, \  \ 0<t<1,
\end{equation}

\begin{equation}\label{eq-65}
u(0)=u(\frac{1}{2}), \  \ u(1)=\int_{0}^{\frac{1}{2}}u(s)ds,
\end{equation}
where $\lambda=799$. Set $\beta=1$, $\alpha=1$, $\eta=1/2$, $T=1$, $a(t)=\frac{6}{25}t$, $f(u)=u(1+\frac{\lambda}{1+u^{2}})$. We can show that $0<\alpha=1<8=2T/{\eta^{2}}$, $0<\beta=1<7/5=(2T-\alpha \eta ^{2})/(\alpha \eta^{2}-2\eta+2T)$. Since $\gamma=\min\{\eta/T, (\alpha(\beta+1)\eta^{2})/2T, (\alpha(\beta+1)\eta(T-\eta))/(2T-\alpha(\beta+1)\eta^{2})\}=\min\{1/2, 1/4, 1/3\}=1/4$, $\Lambda_{1}=((2T-\alpha\eta^{2})-\beta(\alpha\eta^{2}-2\eta+2T))/((2(\beta+1)+T^{-1}\beta\eta(\alpha\eta+2)+\alpha\beta
T)\int_{0}^{T}T(T-s)a(s)ds)=2$, $\Lambda_{2}=((2T-\alpha\eta^{2})-\beta(\alpha\eta^{2}-2\eta+2T))/(2\gamma\eta\int_{\eta}^{T}(T-s)a(s)ds)=100$, and $f_{0}=1+\lambda=800$, $f_{\infty}=1$. Taking $\theta_{1}\in(1/2, 1]$, $\theta_{2}\in[1, 2)$, thus $f_{0}\in((\frac{\theta_{2}}{\gamma})\Lambda_{2},\infty) $, $f_{\infty}\in(0, \theta_{1}\Lambda_{1})$, wich imply {\rm (H7)} and {\rm (H8)} hold. Therefore, by Corollary \ref{cor 4.3}, the BVP \eqref{eq-64}, \eqref{eq-65} has at least one positive solution.
\end{exmp}


\begin{thebibliography}{9}



\bibitem{Ander1}
D. R. Anderson; Nonlinear triple-point problems on time scales,
  \emph{Electron. J. Diff. Eqns.,} \textbf{47} (2004), 1--12.

\bibitem{Ander2}
D. R. Anderson; Solutions to second order three-point problems on time
  scales, \emph{J. Difference Equ. Appl.,} \textbf{8} (2002), 673--688.

\bibitem{Cheng}
Z. Chengbo; Positive solutions for semi-positone three-point boundary
  value problems, \emph{J. Comput. Appl. Math.,} \textbf{228} (2009), 279--286.

\bibitem{Feng}
W. Feng, J. R. L. Webb; Solvability of a three-point nonlinear boundary value
  problem at resonance, \emph{Nonlinear Analysis TMA,} \textbf{30(6)}(1997), 3227--3238.


\bibitem{Gupt}
C.P. Gupta; Solvability of a three-point nonlinear boundary value problem
  for a second order ordinary differential equations , \emph{J. Math. Anal.
  Appl.,} \textbf{168} (1992), 540--551.

\bibitem{Guo}
Y. Guo, W. Ge; Positive solutions for three-point boundary value problems
  with dependence on the first order derivative, \emph{J. Math. Anal. Appl.,}
  \textbf{290} (2004), 291--301.


\bibitem{Han}
X. Han; Positive solutions for a three-point boundary value problem,
  \emph{Nonlinear Analysis TMA,} \textbf{66} (2007), 679--688.

 \bibitem{He}
X. He, W. Ge; Triple solutions for second order three-point boundary
  value problems, \emph{J. Math. Anal. Appl.,} \textbf{268} (2002), 256--265.


\bibitem{Haddou}
F. Haddouchi, S. Benaicha; Positive solutions of nonlinear three-point
  integral boundary value problems for second-order differential equations,
  \url{[http://arxiv.org/abs/1205.1844]}.


\bibitem{Ilin}
 V. A. Il'in, E. I. Moiseev; Nonlocal boundary-value problem of the first kind
  for a Sturm-Liouville operator in its differential and finite difference
  aspects, \emph{Differ. Equ.,} \textbf{23(7)} (1987), 803--810.


\bibitem{Krasn}
M. A. Krasnoselskii; Positive Solutions of Operator Equations, P. Noordhoff, Groningen, The Netherlands, 1964.

\bibitem{Li}
J. Li, J. Shen; Multiple positive solutions for a second-order
  three-point boundary value problem, \emph{Appl. Math. Comput} \textbf{182}(2006),
  258--268.


\bibitem{Luo}
H. Luo, Q. Ma; Positive solutions to a generalized second-order
  three-point boundary-value problem on time scales, \emph{Electron. J. Diff.
  Eqns.,} \textbf{17} (2005), 1--14.

\bibitem{Liang1}
S. Liang, L. Mu; Multiplicity of positive solutions for singular
  three-point boundary value problems at resonance, \emph{Nonlinear Analysis
  TMA,} \textbf{71} (2009), 2497--2505.


  \bibitem{Liu1}
B. Liu; Positive solutions of a nonlinear three-point boundary value
  problem, \emph{Appl. Math. Comput.,}  \textbf{132} (2002), 11--28.


  \bibitem{Liu2}
B. Liu, L. Liu, Y. Wu; Positive solutions for singular second order
  three-point boundary value problems, \emph{Nonlinear Analysis TMA,} \textbf{66} (2007), 2756--2766.

  \bibitem{Liu3}
B. Liu; Positive solutions of a nonlinear three-point boundary value
  problem, \emph{Comput. Math. Appl.,} \textbf{44} (2002), 201--211.

 \bibitem{Liang2}
R. Liang, J. Peng, J. Shen; Positive solutions to a generalized second
  order three-point boundary value problem, \emph{Appl. Math. Comput.,} \textbf{196} (2008), 931--940.

\bibitem{Ma1}
R. Ma; Existence theorems for a second order three-point boundary value
  problem, \emph{J. Math. Anal. Appl.,} \textbf{212} (1997), 430--442.

\bibitem{Ma2}
R. Ma; Multiplicity of positive solutions for second-order three-point
  boundary value problems, \emph{Comput. Math. Appl.,} \textbf{40} (2000), 193--204.

 \bibitem{Ma3}
R. Ma; Positive solutions for a nonlinear three-point boundary value
  problem, \emph{Electron. J. Diff. Eqns.,}  \textbf{34} (1999), 1--8.

\bibitem{Ma4}
R. Ma; Positive solutions for second-order three-point boundary value
  problems, \emph{Appl. Math. Lett.,} \textbf{14} (2001), 1--5.


\bibitem{Ma5}
R. Ma, H. Wang; Positive solutions of nonlinear three-point boundary
  value problem, \emph{J. Math. Anal. Appl.,} \textbf{279} (2003), 216--227.

\bibitem{Maran}
S. A. Marano; A remark on a second order three-point boundary value
  problem, \emph{J. Math. Anal. Appl.,} \textbf{183}(1994), 581--522.


\bibitem{Pang}
H. Pang, M. Feng, W. Ge; Existence and monotone iteration of positive
  solutions for a three-point boundary value problem, \emph{Appl. Math. Lett.,}
  \textbf{21} (2008), 656--661.



\bibitem{Sun1}
H. R. Sun, W. T. Li; Positive solutions for nonlinear three-point boundary
  value problems on time scales, \emph{J. Math. Anal. Appl.,} \textbf{299} (2004), 508--524.

\bibitem{Sun2}
Y. Sun, L. Liu, J. Zhang, R. P. Agarwal; Positive solutions of singular
  three-point boundary value problems for second-order differential equations,
  \emph{J. Comput. Appl. Math.,} \textbf{230} (2009), 738--750.


\bibitem{Tarib}
J. Tariboon, T. Sitthiwirattham; Positive solutions of a nonlinear
  three-point integral boundary value problem, \emph{Bound. Val. Prob.,} ID 519210, doi:10.1155/2010/519210 (2010), 11 pages.


\bibitem{Webb1} J. R. L. Webb; A unified approach to nonlocal boundary value problems,
\emph{Dynam. Systems Appl.,} 5 (2008), 510-515.

\bibitem{Webb2} J. R. L. Webb; Solutions of nonlinear equations in cones and positive
linear operators, \emph{J. Lond. Math. Soc., (2)} 82 (2010), 420-436.

\bibitem{Webb3} J. R. L. Webb, G. Infante; Positive solutions of nonlocal boundary value
problems involving integral conditions, \emph{NoDEA Nonlinear Differential Equations Appl.,} 15 (2008), 45-67.

\bibitem{Webb4} J. R. L. Webb, G. Infante; Positive solutions of nonlocal boundary value
problems: a unified approach, \emph{J. Lond. Math. Soc., (2)} 74 (2006), 673-693.

\bibitem{Webb5} J. R. L. Webb, G. Infante; Non-local boundary value problems of arbitrary order, \emph{J. Lond.
    Math. Soc., (2)} 79 (2009), 238-258.

\bibitem{Webb6} J. R. L. Webb; Positive solutions of a boundary value problem with integral boundary condition,
\emph{ Electron. J. Diff. Eqns.,} \textbf{55} (2011), 1--10.

\bibitem{Webb7} J. R. L. Webb; Nonexistence of positive solutions of nonlinear boundary value problems,
\emph{ Electron. J. Qual. Theory Differ. Equ.,} \textbf{61} (2012), 1--21.


\bibitem{Xu}
X. Xu; Multiplicity results for positive solutions of some semi-positone
  three-point boundary value problems, \emph{J. Math. Anal. Appl.,} \textbf{291} (2004), 673--689.



\end{thebibliography}
\end{document}